\documentclass[11pt]{amsart}

\usepackage{tikz}
\usetikzlibrary{arrows.meta}
\newtheorem{theorem}{Theorem}[section]
\newtheorem{lemma}[theorem]{Lemma}

\newtheorem{corollary}[theorem]{Corollary}
\newtheorem{fact}[theorem]{Fact}
\theoremstyle{remark}

\theoremstyle{definition}

\numberwithin{equation}{section}

\begin{document}

% \title[short text for running head]{full title}
\title[diffeomorphisms on manifolds]{Anti-classification results for conjugacy of diffeomorphisms on manifolds}

%    Only \author and \address are required; other information is
%    optional.  Remove any unused author tags.

%    author one information
% \author[short version for running head]{name for top of paper}
\author{Bo Peng}
\address{McGill University}
\curraddr{Department of Mathmatics and Statistics, McGill University. 805 Sherbrooke Street West Montreal, Quebec, Canada, H3A 2K6}
\email{bo.peng3@mail.mcgill.ca}
\thanks{}

%    \subjclass is required.
%\subjclass[2010]{Primary }

\date{}

\dedicatory{}

%    Abstract is required.
\begin{abstract}
We show that the topological conjugacy relation of diffeomorphisms on any manifold of dimension at least 2 is not classifiable by countable structures. This answers a question of Foreman and Gorodetski. We also prove that $E_0$ is reducible into the topological conjugacy relation of minimal diffeomorphisms on the 2-torus, which answers a question of Foreman.
\end{abstract}

\maketitle
\section{Introduction}

In the recent 40 years, complexity theory for equivalence relations has been developed by Kechris, Hjorth, Louveau and others (see \cite{Gaobook} and \cite{Kecbook}). Let $E$ and $F$ be two equivalence relations on Polish spaces $X$ and $Y$, respectively. A Borel function from $X$ to $Y$ is called a \textbf{Borel reduction} from $E$ to $F$, if for any $x_1,x_2\in X$, we have
$$
x_1Ex_2\,\,\mbox{if and only if}\,\,f(x_1)Ff(x_2).
$$
We say $E$ is \textbf{Borel reducible} to $F$ if there is a Borel reduction from $E$ to $F$ and denote by $E\leq_B F$. If an equivalence relation $E$ is Borel reducible to $F$ then we regrad $F$ as a more complicated equivalence relation. For many classification problems appearing in mathematics, their exact Borel complexity has been computed (see \cite{Sabokcomplete}, \cite{Zielinski}, \cite{Torsion-free}).

Common measure of complexity is connected with the following questions:

\begin{enumerate}
    \item Is the equivalence relation \textbf{smooth}?
    \item Is the equivalence relation Borel \textbf{classifiable by countable structures}?
    \item Is the equivalence relation \textbf{Borel}?
\end{enumerate}

Note that if the answer to either $(2)$ or $(3)$ is negative, then the answer to $(1)$ is also negative. If the answer to some of the above questions is negative, then this gives a rigorous proof that the classification is impossible and we call such a result an anti-classification theorem. 

Well-studied examples of classification problems in dynamical systems include conjugacy problems in ergodic theory (see \cite{Foremanturbulent}, \cite{Forremannoborel}, \cite{Kundetimechange}), in symbolic dynamics (see \cite{bowen's}, \cite{MarcinToeplitz}, \cite{KayaToeplitz}), in Cantor systems (see \cite{NonBorelCantorminimal}, \cite{Kayapointed}, \cite{Vejnartransitive}) and on the circle (see 
\cite{Kundesmoothconjugacy}, \cite{Vejnaronedim}). 

 In the 1960s, Smale \cite{ICM, Smale} proposed a program classifying diffeomorphism on smooth manifold by topological conjugacy. Foreman and Gorodetski made significant contributions to Smale's program. In \cite{ForemanGoropaper}, they proved that for a manifold $M$ with dimension $n$, the topological conjugacy relation of diffeomorphisms on $M$ is 
\begin{enumerate}
    \item not smooth if $n\geq 2$;
    \item not Borel if $n\geq 5$.
\end{enumerate}

Recently, Foreman and Weiss, and independently Vejnar generalized non-Borelness to all manifolds.

In \cite[Question 4]{ForemanGoropaper}, Foreman and Gorodetski ask if the topological conjugacy relation of diffeomorphism on a manifold reducible to an $S_{\infty}$ action? 

We answer this question as follows:

\begin{theorem} \label{Main 1}
    Let $M$ be a manifold of dimension greater than equal to $2$, then the topological conjugacy relation of $C^{\infty}$-diffeomorphisms on $M$ is not classifiable by countable structures. 
\end{theorem}

Our proof builds on the work of Gorodetski and Foreman \cite{ForemanGoropaper} on non-smoothness of the the topological conjugacy relation of $C^{\infty}$-diffeomorphisms.

In \cite{Hjorthbook} Hjorth proved that the topological conjugacy relation of homeomorphisms on the unit interval is classifiable by countable structures and the same argument works for the circle. This leads to the following corollary.
\begin{corollary}
  Let $M$ be a compact manifold of dimension $n$, then the topological conjugacy relation of diffeomorphisms on $M$ is classifiable by countable structures if and only if $n=1$. 
\end{corollary}

In \cite[Open Question 14]{Foremansurvey}, Foreman asks if $E_0$ is Borel reducible to the collection of topologically minimal diffeomorphisms of the 2-torus with the relation of topological conjugacy?

We give an affirmative answer to this question. In fact, we prove a stronger result.

\begin{theorem} \label{Main2}
    The topological conjugacy relation of minimal rotations on the $2$-torus is not amenable.
\end{theorem}

%\begin{theorem} \label{Main3}
    %For any manifold with dimensional greater than equal to $2$, the smooth %conjugacy relation of diffeomorphism is not classifiable by countable structures.
%\end{theorem}

%Section 2 is preliminaries and Theorems we will use. In Section \ref{MainTheorem2}, we will prove Theorem \ref{Main2}. In Section 3, we will define some functions which will be used in the next section. In Section \ref{proof of reduction} we will show the construction and the prove of Theorem \ref{Main 1}. In Section \ref{Smooth conjugacy}, we will prove Theorem \ref{Main3}.
\section{preliminaries}
 Let $(X,d)$ be a compact metric space and $A\subset X$ is a closed subset. Let 
 $$
|A|={\rm sup}\{d(x,y),\,x,y\in A\}.
 $$
Denote by ${\rm Homeo}(X)$ the group of homeomorphisms on $X$. For $f,g\in {\rm Homeo}(X)$, we will use the notation:

$$
||f-g||={\rm sup}_{x\in X}d(f(x),g(x)).
$$
$(X,\mu)$. 

We say $E$ is \textbf{smooth} if $E\leq_B =_{\mathbb{R}}$. We say that $E$ is \textbf{classifiable by countable structures} if $E$ is Borel reducible to a Borel action of the group $S_\infty$.
Let $E$ be a countable Borel equivalence relation on a probability space  $(X,\mu)$.  $E$ is $\mu$-\textbf{amenable} if there is a sequence of Borel functions 
$$
\phi_n:E\rightarrow \mathbb{R}_{\geq0}
$$

such that
\begin{enumerate}
    \item for any $x\in X$, $\sum_{yEx} \phi_n(x,y)=1$, and
    \item there is a Borel $E$-invariant set $A\subset X$ with $\mu(A)=1$ such that for all $(x,x')\in E \cap(A\times A)$,
    $$
       \sum_{yEx}|\phi_n(x,y)-\phi_n(x',y)| \rightarrow0\,\,\mbox{as}\,\,n\rightarrow \infty.
    $$
\end{enumerate}

We say $E$ is \textbf{amenable} if $E$ is $\mu$-amenable for all Borel probability measure $\mu$. 

Let $(X,\mu)$ be a standard probability space. Denote by ${\rm MPT}(X,\mu)$ the space of measure preserving Borel transformations on $(X,\mu)$.

\begin{theorem} \label{amenable}
    $($\cite[Proposition 7.4.6]{Gaobook}$)$ Let $G$ be a countable group, $(X,\mu)$ is a standard probability space. If $G \subset {\rm MPT}(X, \mu)$ and the orbit equivalence relation $E^X_G$ is $\mu$-amenable, then $G$ is an amenable group.
\end{theorem}

Now we look at the notion of turbulence developed by Hjorth (See \cite{Hjorthbook} for more details). Let $G$ be a Polish group acting on a Polish space $X$, for $x\in X$, take open sets $V\subset X$ and $U\subset G$ with $x\in V$ and $1\in U$, we define the local-$V$-$U$ orbit of $x$ to be
    $$
    \{ g_1\circ g_2 \circ \cdots \circ g_m x|\,\,g_i\in U\}.
    $$
    We say the action is \textbf{turbulent} if

\begin{enumerate}
    \item Every orbit is dense.
    \item Every orbit is meager.
    \item For any $x\in X$, any local-$V$-$U$ orbit of $x$ is somewhere dense. 
\end{enumerate}

If the action of $G$ on $X$ is turbulent, then $E^X_G$ is not classifiable by countable structures \cite{Hjorthbook}. We will use the following fact.

\begin{fact}
    $($\cite[Lemma 6.2.2]{Kanoveibook}$)$ $\mathbb{R}^{\omega}/c_0$ is a turbulent action and $\mathbb{R}^{\omega}/c_0$ is Borel bireducible with $[0,1]^{\omega}/c_0$.
\end{fact}

\section{Minimal diffeomorphism on 2-torus} \label{MainTheorem2}

The proof of Theorem \ref{Main2} is based on several dynamical systems, group theory and descriptive set theory facts. We will only prove the 2-torus case, the same argument works for $n$-torus for every $n\geq 2$.

%\textbf{Dynamical system facts:}
%\begin{enumerate}
    %\item A rotation $T_{\alpha,\beta}:(x,y)\rightarrow(x+\alpha,y+\beta)$ on 2-Torus is minimal if and only if $\alpha$, $\beta$ and $1$ are rationally independent.
    %\item (Folklore, formally \cite[Proposition 1.5]{P})Two rotations $T_{\alpha,\beta}$ and $T_{\alpha',\beta'}$ are topologically conjugate if and only if there is an $A\in {\rm GL(2,\mathbb{Z})}$ such that $A(\alpha,\beta)^T=(\alpha',\beta')^T$
%\end{enumerate}

%\textbf{Descriptive set theory facts:}
%\begin{enumerate}
    %\item Let $G$ be a countable non-amenable group acting on a probability space $(X,\mu)$. If the action of $G$ is $\mu$-preserving and free $\mu$-almost everywhere, then the equivalence relation $E_G^X$ is not amenable.

    %\item (Harrington-Kechris-Louveau Theorem) Let $E$ be a Borel equivalence relation on a Polish space $X$, then either $E$ is smooth or $E_0$ is continuous reducible to $E$.

    %\item Let $E$ be a countable Borel equivalence relation then $E$ is smooth implies $E$ is amenable.
%\end{enumerate} 

%\textbf{Group theory facts}
%\begin{enumerate}
    %\item The group  ${\rm GL(2,\mathbb{Z})}$ is not amenable.

    %\item The action of ${\rm GL(2,\mathbb{Z})}$ on $(\mathbb{R}/\mathbb{Z})^2$ is Lebesgue measure preserving.
%\end{enumerate}

We will prove that the topological conjugacy relation of minimal rotations on the 2-torus is not amenable, thus $E_0$ is reducible to it by the Harrington--Kechris--Louveau theorem \cite{Dictheorem}.

Let $T_{\alpha,\beta}$ be the rotation on the 2-torus defined as:
$$
T_{\alpha,\beta}(x,y)=(x+\alpha,y+\beta)
$$
where $(x,y),(\alpha,\beta)\in (\mathbb{R}/\mathbb{Z})^2$.
\begin{lemma} \label{nonamenable}
    Let
    $$
    X=\{(\alpha,\beta)\in (\mathbb{R}/\mathbb{Z})^2| \alpha ,\,\,\beta\,\,\mbox{and}\,\,1\,\,\mbox{are rationally independent}\}
    $$

    then the natural action of ${\rm GL(2,\mathbb{Z})}$ on $X$ is not amenable.
\end{lemma}
\begin{proof}
   Denote by $\lambda^2$ the Lebesgue measure on $(\mathbb{R}/\mathbb{Z)}^2$. Since for conull many $x\in (\mathbb{R}/\mathbb{Z)}$ (irrationals), the section of $x$ in 
  $X$ is the complement of a countable set, by Fubini's Theorem, we know that $\lambda^2(X)=1$. 

   The action of ${\rm GL(2,\mathbb{Z})}$ on $X$ is free almost everywhere since 
   $A(\alpha,\beta)^T=(\alpha,\beta)^T$ holds for $A\neq I_2$ at most for a one-dimensional eigenspace which has Lebesgue measure $0$. Since the group  $ {\rm GL(2,\mathbb{Z})}$ is not amenable and the action of ${\rm GL(2,\mathbb{Z})}$ on $(\mathbb{R}/\mathbb{Z})^2$ is Lebesgue measure preserving, by Theorem \ref{amenable},  this implies that the action of ${\rm GL(2,\mathbb{Z})}$ on $X$ is not amenable.
   \end{proof}
   Now we send any $(\alpha,\beta)\in X$ to the rotation $T_{\alpha,\beta}$, it is routine to check that this map is continuous. The following fact is standard (see \cite[Proposition 1.5]{PanHou}). We provide a proof for completeness.
\begin{fact}\label{Conjugateiff condition}
    Two minimal rotations $T_{\alpha,\beta}$ and $T_{\alpha',\beta'}$ are topologically conjugate if and only if there exists an $A\in {\rm GL}_2(\mathbb{Z})$ such that $A(\alpha,\beta)^{T}=(\alpha',\beta')^T$.
\end{fact}

  \begin{proof}
          Suppose $h$ is a conjugacy map between $T_{\alpha,\beta}$ and $T_{\alpha',\beta'}$. After doing a shift, we may assume that $h(0,0)=(0,0)$.

          Note that $h$ is a group isomorphism on the subgroup $\{(n\alpha,n\beta)\}_n$ of the 2-torus. But $\{(n\alpha,n\beta)\}_n$ is dense on $(\mathbb{R}/\mathbb{Z})^2$, thus $h$ is an isomorphism of $(\mathbb{R}/\mathbb{Z})^2$. Let $H:\mathbb{R}^2 \rightarrow \mathbb{R}^2$ be the lift of $h$. Since $h$ is a group isomorphism, we know $H$ is a linear isomorphism of $\mathbb{R}^2$. Since $H$ is the lift of $h$, we know $H$ must map $\mathbb{Z}^2$ to $\mathbb{Z}^2$, thus we have $H=A$ for some $A\in {\rm GL}_2(\mathbb{Z})$.
          %Now for any $(x,y)\in \mathbb{R}^2$ with $\pi(x,y)=(a,b)$. We have 
          %$$
           %pi \circ H(x,y)=h\circ\pi(x,y)=h(a,b)={\rm %lim}_nh(\pi(k_n\alpha,k_n\beta))
           %$$
           %$$
           %={\rm lim}_n\pi\circ H(k_n\alpha,k_n\beta)={\rm lim}_n\pi\circ %A(k_n\alpha,k_n\beta)=\pi \circ A(x,y).
          %$$
          %Note that $\pi$ is linear, we have $\pi(H(x,y)-A(x,y))=0$. This means 
          %$H(x,y)-A(x,y)$ is a constant since the image is in $\mathbb{Z}^2$ and %$\pi(H-A)$ is continuous. Take $(x,y)=(0,0)$, we have
          %$$
           %%$$
          %hus, we know $H(x,y)=A(x,y)$ for all $(x,y)\in \mathbb{R}^2$. We know $H$ is a self isometry on $\mathbb{Z}^2$ which means $A\in {\rm GL}_2(\mathbb{Z})$. Thus $H$ is an affine homeomorphism in ${\rm GL}_2(\mathbb{Z})$ sending $(\alpha,\beta)$ to $(\alpha',\beta')$.
      \end{proof}

      Together, Lemma \ref{nonamenable} and Fact \ref{Conjugateiff condition} imply the topological conjugacy relation of minimal rotations on the 2-torus is not amenable. Every non-amenable countable Borel equivalence relation is not smooth and by the Harrington--Kechris--Louveau theorem \cite{Dictheorem}, we know that $E_0$ is reducible to the topological conjugacy of minimal diffeomorphisms on the 2-torus.
      
   %Every non-amenable countable Borel equivalence relation is not smooth. By Harrington-Kechris-Louveau Theorem, we know $E_0$ is reducible to the action of ${\rm GL(2,\mathbb{Z})}$ on $X$. And by Fact \ref{Conjugateiff condition}, we know the action is continuously reducible to topological conjugacy relation of minimal rotations on 2-torus. Thus, $E_0$ is Borel reducible to conjugacy relation of minimal rotations which are minimal diffeomorphisms.

\section{Auxiliary functions} \label{functions define}
    This section is a preparation for Section \ref{proof of reduction}. We will define two kinds of functions. One is for the construction of our diffeomorphisms, the other is for proving conjugacy.  From now on we fix a natural number $n\geq 2$ to be the dimension of a give manifold.
    
    \textbf{Functions of the first type}:
    
    Given $a\in [0,1]^d, r>0$ and $q\geq 0$, we will define a function $\psi(a,r,q)$ on the disk $D(a,r)$ (an $n$ dimensional disk centered at $a$ with diameter $r$) as follows. 

    First, we pick a $C^{\infty}$ function $\phi$ from $[0,1]$ to $[0,1]$ such that

\begin{enumerate}
    \item $\phi$ is non-increasing;
    \item $\phi[0,1/3]\equiv 1$;
    \item $\phi[2/3,1]\equiv 0$.
\end{enumerate}
Below is such a function:

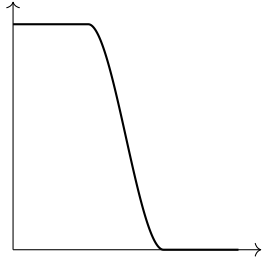
\begin{center}
\begin{tikzpicture}[scale=3]

% Axes (no labels, no ticks)
\draw[->, thin] (0,0) -- (1.1,0); % x-axis
\draw[->, thin] (0,0) -- (0,1.1); % y-axis

% Smooth transition function (avoids "dimension too large" error)
\draw[thick, black, smooth, samples=200, domain=0:1] plot (\x, {
    \x < 1/3 ? 1 : (
    \x > 2/3 ? 0 : 
    % Safer smooth transition: cubic Hermite interpolation
    1 - ( (3*\x - 1)^2 * (3 - 2*(3*\x - 1)) )
    }
);

\end{tikzpicture}
\end{center}

Next define $\psi(a,r,q)$ on $D(a,r)$ to be the diffeomorphism and let $(\rho,\theta_1,\dots,\theta_n)$ be the spherical coordinates for a point on the $d$-disk $D(a,r)$.

$$
\psi(a,r,q)((\rho, \theta_1,\dots,\theta_{n-1}))=(\rho,\theta_1,\dots, \theta_{n-1}+2q\phi(\frac{\rho}{r})\pi)
$$

where $0\leq \rho\leq r$, $\theta_i\in [0,\pi)$ for $1\leq i\leq n-2$, $\theta_{n-1}\in [0,2\pi)$.

\textbf{Functions of the second type}:

   Let $x,y\in [0,1]^d$ with $x_{i}=y_i$ for all $1\leq i\leq n-1$ and $r>0$. Let $R=\Pi_{1\leq i\leq n}[a_i,b_i]\subset [0,1]^d$ with the property that both $D(x,r),D(y,r)$ are in the interior of $R$.  We will define a function $m(R,x,y,r):[0,1]^d\rightarrow [0,1]^d$ with the property:
\begin{enumerate}
    \item For all $t\in D(x,r)$, we have $m(R,x,y,r)(t)=t+(y-x)$. 
    \item $m(R,x,y,r)$ is identity on the boundary of $R$ and the rest part of $[0,1]^d$  
\end{enumerate}

Below is a picture for $m(R,x,y,r)$ for the two dimensional case. The function moves the top square centered at $x$ to the bottom square centered at $y$. $r$ is the diameter of the circle.
\\

\begin{center}
\begin{tikzpicture}
    % Large rectangle
    \draw (0,0) rectangle (4,6);  % Adjusted dimensions for vertical layout
    
    % Two small circles stacked vertically
    \draw (2,4.5) circle (0.5);   % Top circle (center at (2,4.5), radius 0.5)
    \draw (2,1.5) circle (0.5);    % Bottom circle (center at (2,1.5), radius 0.5)
    
    % Vertical arrow from top circle to bottom circle
    \draw[-{Stealth[scale=1.2]}, thick] (2,4) -- (2,2);
    
    % Optional labels (can be removed)
    \node at (2,5.2) {Top Circle};
    \node at (2,0.8) {Bottom Circle};
    \node at (2, -0.7) {The function $m(R,x,y,r)$};
    \node at (0.3,5.8) {$R$};
\end{tikzpicture}
\end{center}

\section{Topological conjugacy of diffeomorphisms on $M$ with ${\rm dim}(M)\geq 2$} \label{proof of reduction}

In this section we prove the following:
\begin{theorem}
    Let $M$ be a manifold with ${\rm dim}(M)\geq 2$. Then $[0,1]^{\omega}/c_0$ is Borel reducible to the topological conjugacy relation of $C^{\infty}$-diffeomorphisms on $M$.
\end{theorem}
\begin{proof}
Let $d={\rm dim}(M)$. We will reduce $[\frac{1}{4},\frac{3}{4}]^{\omega}/c_0$ to the conjugacy relation of diffeomorphism on $[0,1]^d$ with the property that on the boundary of $[0,1]^d$, the diffeomorphism is identity. For a manifold $M$, we can cut a window $[0,1]^d$ from $M$ and paste our diffeomorphism cubes on it. All disks are $d$-disks in $[0,1]^d$ in this section.

Let $H_n$ be the rectangle $[\frac{1}{n+1},\frac{1}{n}]\times [0,1]^{d-1}$ and let $k_n$ be $(\frac{1}{n}+\frac{1}{n+1})/{2}$.

Now, let $\alpha\in [\frac{1}{4},\frac{3}{4}]^\omega$. Fix a countable dense subset $\{p_n\}$ of $[\frac{1}{4},\frac{3}{4}]$.

Let $a$ be a point in $[0,1]^d$. Denote by $D(a,r)$ a $d$-disk centered at $a$ with diameter $r$. 

For $\alpha \in [\frac{1}{4},\frac{3}{4}]^\omega$, we will choose a sequence of points $a^{\alpha}_n\in H_n$ and define a sequence of disks centered at $a^{\alpha}_n$. 

Take $r_n<(\frac{1}{n}-\frac{1}{n+1})/{2}$. Note that the choice of $r_n$ is independent of the choice of $\alpha \in [\frac{1}{4},\frac{3}{4}]^{\omega}$.

 Denote by $D(\alpha)_{2n}$ the  $d$-disk $D(a^{\alpha}_{2n},r_{2n})$ with $a^{\alpha}_{2n}=(k_{2n},\dots,\frac{1}{2},\alpha(n))$. Denote by $D(\alpha)_{2n+1}$ the $d$-disk $D(a^{\alpha}_{2n+1},r_{2n+1})$ with $a^{\alpha}_{2n+1}=(k_{2n+1},\dots,\frac{1}{2},p_n)$. 
 
 \begin{lemma}\label{inside}
     For $\alpha \in [\frac{1}{4},\frac{3}{4}]^{\omega}$, we have $D(\alpha)_n\subset H_n$.
 \end{lemma}
 \begin{proof}
     Since $\alpha \in [\frac{1}{4},\frac{3}{4}]^{\omega}$ and clearly $r_n<\frac{1}{4}$, thus in the last $n-1$ coordinates of $a^{\alpha}_n$ are inside $H_n$ if we add or subtract $r_n$. For the first coordinate, this is true since $r_n<(\frac{1}{n}-\frac{1}{n+1})/{2}$ and all other coordinates of $a^{\alpha}_n$ are equal to $\frac{1}{2}$ which is in the middle of $H_n$.
 \end{proof}
 Here is a picture of those disks in the 2 dimensional case:

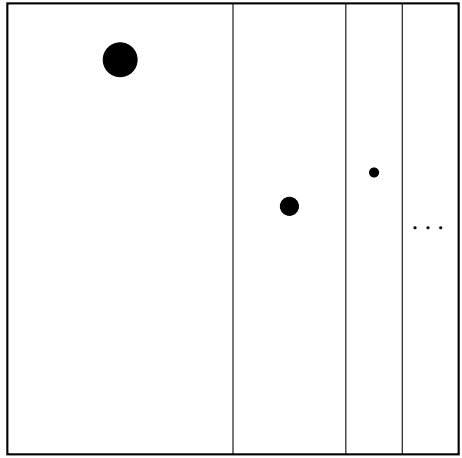
\begin{center}

\begin{tikzpicture}[scale=1.5]
    % Draw the main square 
   
    \draw[thick] (0,0) rectangle (4,4);
    
    % Draw three vertical divisions
    \draw (2,0) -- (2,4);    % First division
    \draw (3,0) -- (3,4);    % Second division
    \draw (3.5,0) -- (3.5,4); % Third division
    \draw (3.75,2) node {$\cdots$}; % Continuation dots
    
    % Middle vertical line
    %\draw[dashed,gray] (2,2) -- (3.75,2);
    
    % Draw decreasing disks in each rectangle
    % First rectangle (width 2)
    \filldraw (1,3.5) circle (0.15);
    %\filldraw (1,3.0) circle (0.10);
    %\filldraw (1,2.5) circle (0.06);
    %\draw (1,2.1) node {$\vdots$};
    
    % Second rectangle (width 1)
    %\filldraw (2.5,2.7) circle (0.12);
    \filldraw (2.5,2.2) circle (0.08);
    %\filldraw (2.5,1.7) circle (0.05);
    %\draw (2.5,1.3) node {$\vdots$};
    
    % Third rectangle (width 0.5)
    %\filldraw (3.25,3.5) circle (0.09);
    %\filldraw (3.25,3.0) circle (0.06);
    \filldraw (3.25,2.5) circle (0.04);
    %\draw (3.25,2.1) node {$\vdots$};
   
\end{tikzpicture}

\end{center}

Now let's start our construction:

 Set $f_{\alpha}$ on the disk $D(\alpha)_n$ to be $\psi(\alpha)(a^{\alpha}_n,r_n,\frac{1}{2^{n+1}})$ where $\psi$ is the type 1 function and on the rest of $[0,1]^d$, take $f_{\alpha}$ to be the identity. Note that $f_{\alpha}$ is a rotation on each disk $D(\alpha)_n$.

\begin{lemma}
    Let $\alpha\in [\frac{1}{4},\frac{3}{4}]^{\omega}$. The function $f_{\alpha}$ is a $C^{\infty}$-diffeomorphism.
\end{lemma}
\begin{proof}
    Fix $\alpha\in [\frac{1}{4},\frac{3}{4}]^{\omega}$.  By definition, $\psi(\alpha)(a^{\alpha}_n,r_n,\frac{1}{2^{n+1}})$ is the identity near the boundary of $D(\alpha)_n$, since the function $\phi=0$ on an interval containing $1$. This implies that $f_{\alpha}$ is differentiable on points which are not on $\{(0,0,\dots,0)\}\times [0,1]$, since $\phi$ is $C^{\infty}$-differentiable thus $\psi(a^{\alpha}_n,r_n,\frac{1}{2^{n+1}})$ is $C^{\infty}$-differentiable.
    
    Now let $p\in \{(0,0,\dots,0)\}\times [0,1]$, denote by $p_i$ the $i^{\rm th}$ coordinate of $p$. Let $f_{\alpha}=(f_1,\dots,f_d)$. Since for $1\leq i<d$, by our construction, $f_i$ is equal to the projection on $i^{\rm th}$ coordinate, thus we know it is differentiable. For $f_d$ and  $1<i\leq d$, the partial derivative $\frac{\partial f_d}{\partial x_i}$ is just taking derivative of the projection function of the last coordinate  which is $C^{\infty}$-differentiable. Now we prove that $\frac{\partial f_d}{\partial x_1}$ exists, let $p'$ be a point in any neighborhood of $p$, we may assume that $p'\in H_m$. Note that the distance of $p$ on the first coordinate to $0$ is greater than $\frac{1}{m+1}$ and in the last coordinate $f_{\alpha}$ is a rotation with angle at most $\frac{2\pi}{2^{m+1}}$, we have the fraction 
    $$
    \frac{|(f_d)(p)-(f_d)(p')|}{|p_1-p'_1|}<(m+1)\frac{2\pi}{2^{m+1}}r_m
    $$ 
    will converge to $0$ as $p'$ approaching $p$. In higher level of derivatives, this limit will also go to $0$ since exponential decreasing property are preserved under taking derivatives. Thus, $f_{\alpha}$ is $C^{\infty}$-differentiable on $\{(0,0,\dots,0)\}\times [0,1]$, we have $f_{\alpha}$ is a $C^{\infty}$-diffeomorphism. 
\end{proof}

 \begin{lemma}\label{periodic point}
     Let $\alpha\in [\frac{1}{4},\frac{3}{4}]^{\omega}$ and $n\in \mathbb{N}$.  The point $a^{\alpha}_{n}$ is the only point in $[0,1]^d$ which has a punctured neighborhood consisting only of $2^{n+1}$-periodic points with respect to $f_{\alpha}$.
 \end{lemma}

   \begin{proof}
       Note that by definition of $\phi$ and $\psi$, we have two types of fixed points
       \begin{enumerate}
           \item  $a^{\alpha}_n$,
           \item  fixed points which are not on $D(a^{\alpha}_n, \frac{1}{3}r_n)$.
       \end{enumerate}
       For the second type of fixed points, all of their punctured neighborhoods contain another fixed point.
   \end{proof}
\begin{lemma} \label{pointed conjugacy}
    Let $\alpha,\beta\in [1/4,3/4]^{\omega}$. Suppose two systems $([0,1]^d, f_{\alpha})$ and $([0,1]^d, f_{\beta})$ are topologically conjugate by $h$. Then $h(a^{\alpha}_n)=a^{\beta}_n$ for any $n\in \mathbb{N}$.
\end{lemma}

\begin{proof}
    By Lemma \ref{periodic point},  $a^{\alpha}_n$ is the only fixed point which has a punctured neighborhood consisting only of $2^{n+1}$-periodic points with respect to $f_{\alpha}$. And $a^{\beta}_n$ is the only point in $([0,1]^d,f_{\beta})$ with such a property. 
\end{proof}

\begin{lemma}
    If $\alpha c_0 \beta$, then $f_{\alpha}$ and $f_{\beta}$ are topologically conjugate.
\end{lemma}
\begin{proof}
    We will define a sequence of homeomorphism $h_n: H_n \rightarrow H_n$ and get the conjugacy map $h$ by pasting $h_n$ together.
    Let $R_{n}$ be a rectangle in $H_{2n}$ such that both $D(\alpha)_{2n}$ and $D(\beta)_{2n}$ are in the interior of $R_n$ . Also, since the distance of two centers of the disks $D(\alpha)_{2n}$ and $D(\beta)_{2n}$ is $|\alpha(n)-\beta(n)|$ we can assume that  
    $$
    |R_n|<4r_{2n}+2|\alpha(n)-\beta(n)|
    $$ 
    
    This upper bound is to make sure the limit of $|R_n|$ goes to $0$ as $n$ goes to infinity.
    
    Take $h_{2n}$ to be $m(R_n, a^{\alpha}_{2n}, a^{\beta}_{2n},r_{2n})$ on $R_n$. On the rest part of $H_n$, take $h_n$ to be identity. 

    Take $h_{2n+1}$ to be identity for all $n$.

    Since $h_n$ is identity on the boundary and outside $R_n$, we know $||h_n-{\rm Id}||\leq |R_n|$. Since $\alpha$ is $c_0$ equivalent with $\beta$, we know $|R_n|$ converges to $0$. Thus on  $\{(0,0,\dots,0)\}\times [0,1]$, we can define $h$ to be the identity. $h$ is clearly a homeomorphism and moving $D(\alpha)_{2n}$ to $D(\beta)_{2n}$ on each $H_{2n}$ which is a conjugacy map. 
\end{proof}

\begin{lemma}
    If $\alpha$ is not $c_0$ equivalent with $\beta$, then $f_{\alpha}$ and $f_{\beta}$ are not conjugate. 
 \end{lemma}
 \begin{proof}
    Since $\alpha$ and $\beta$ are not $c_0$ equivalent. We can find subsequence $\alpha(n_k)$ converging to $a$ and $p_{m_k}$ converging to $a$ while $\beta(n_k)$ converging to $b$ with $a\neq b$. 

    By Lemma \ref{pointed conjugacy}, suppose there is a conjugacy map $h$, we have $h(a^{\alpha}_n)=a^{\beta}_n$. Let $x_{2k}=a^{\alpha}_{n_k}$ and $x_{2k+1}=a^{\alpha}_{m_k}$. Also, let $y_{2k}=(a^{\beta}_{n_k})$ and $y_{2k+1}=a^{\beta}_{m_k}$. Then $h$ is a homeomorphism maping $x_n$ to $y_n$. But $(x_n)$ is convergent while $(y_n)$ is not. This is a contradiction.
 \end{proof}
\end{proof}
\textbf{Acknowledgement}: The author would like to thank Marcin Sabok for suggesting using the relation $c_0$ and for providing proof of Fact \ref{Conjugateiff condition}. 
%\section{Smooth conjugacy}\label{Smooth conjugacy}

%In this section we will look at smooth conjugacy relation of diffeomorphisms on $I^2$. Note that, similarly, this method can be easily generalized to higher dimensions. 

%The strategy is a little different here. The equivalence $c_0$ is not enough here. And in general, for a Polish group action $E^{\mathbb{R}^{\omega}}_G$, we don't know if we can restricted it on the Hilbert cube. To keep complexity, we will look at a turbulent action on $(S^1)^{\omega}$.

%First, 

%    Text of article.

%    Bibliographies can be prepared with BibTeX using amsplain,
%    amsalpha, or (for "historical" overviews) natbib style.
\bibliographystyle{plain}
 % We choose the "plain" reference style
\bibliography{bibliography}
\end{document}